\documentclass[10pt]{amsart}

\usepackage{amsmath}
\usepackage{amssymb}
\usepackage{amscd}
\usepackage{latexsym}
\usepackage[latin1]{inputenc}
\usepackage{enumerate}
\usepackage[all]{xy}

\usepackage{graphicx}
\usepackage{changebar,color}

\newtheorem{theorem}{Theorem}[section]
\newtheorem{corollary}[theorem]{Corollary}
\newtheorem{lemma}[theorem]{Lemma}

\newtheorem{proposition}[theorem]{Proposition}

\newtheorem{definition}[theorem]{Definition}

\numberwithin{equation}{section}

\newcommand{\Ker}{\operatorname{Ker}}

\def\P{\mathbb{P}}

\def\Z{\mathbb{Z}}

\def\Q{\mathbb{Q}}

\def\C{\mathbb{C}}

\newcommand{\ra}{\rightarrow}

\date{\today}

\title{Frobenius distribution for pairs of elliptic curves and exceptional isogenies}
\author{Fran\c cois Charles}
\address{Fran\c cois Charles, Department of Mathematics, Massachusetts Institute of Technology, Cambridge, MA 02139-4307, USA; and Laboratoire de mathématiques d'Orsay, UMR 8628 du CNRS, Universit{é}
Paris-Sud,
B{â}timent 425, 91405 Orsay cedex, France}
\email{francois.charles@math.u-psud.fr}
\begin{document}

\begin{abstract}
Let $E$ and $E'$ be two elliptic curves over a number field. We prove that the reductions of $E$ and $E'$ at a finite place $\mathfrak p$ are geometrically isogenous for infinitely many $\mathfrak p$, and draw consequences for the existence of supersingular primes. This result is an analogue for distributions of Frobenius traces of known results on the density of Noether-Lefschetz loci in Hodge theory. The proof relies on dynamical properties of the Hecke correspondences on the modular curve.
\end{abstract}

\maketitle

\section{Introduction}

The goal of this paper is to investigate the behavior, in a particular example, of the distribution of the conjugacy classes Frobenius morphisms for a Galois representation of weight $2$ with Hodge-Tate weight $h^{2,0}=1$. Since we are primarily interested in geometric applications, we will state our main result in the context of isogenies of elliptic curves. 

Let us start by recalling a Hodge-theoretic theorem due independently to Green \cite{Voisin2002} and Oguiso \cite{Oguiso03} -- see also \cite{BKPSB}. If $H$ is a Hodge structure of weight $2$, let $\rho(H)$ be the \emph{Picard number} of $H$, that is, the rank of the group of Hodge classes -- integral classes of type $(1,1)$ -- in $H$.

Let $\Delta$ be the unit disk in $\C$, and let $H$ be a non-trivial variation of Hodge structures of weight $2$ over $\Delta$ with Hodge number $h^{2, 0}=1$. Let $M$ be the minimal value of the integers $\rho(H_s)$ for $s$ in $\Delta$. Then the Noether-Lefschetz locus 
$$NL(H):=\{s\in\Delta, \, \rho(H_s)>M\}$$
is dense in $\Delta$. Note however that $\rho(H_s)=M$ if $s$ is very general in $\Delta$.

In this paper, we focus on an arithmetic analogue of the theorem above, where $\Delta$ is replaced by a suitable open subset of the spectrum of the ring of integers in a number field $k$. In that setting, variations of Hodge structures are replaced by representations of the absolute Galois group $G_k$ of $k$, and the Noether-Lefschetz locus is replaced by the set of primes at which some power of the Frobenius has invariants which do not have a finite orbit under the whole group $G_k$. 

Our main theorem is the following. Say that two elliptic curves over a field $k$ are \emph{geometrically isogenous} if they are isogenous over an algebraic closure of $k$. 

\begin{theorem}\label{theorem:main}
Let $k$ be a number field, and let $E$ and $E'$ be two elliptic curves over $k$. Then there exist infinitely many finite places $\mathfrak p$ of $k$ such that the reductions $E_{\mathfrak p}$ and $E'_{\mathfrak p}$ of $E$ and $E'$ modulo $\mathfrak p$ are geometrically isogenous. 
\end{theorem}

If $k$ is the function field of a curve over a finite field and $E, E'$ are both non-isotrivial elliptic curves, the corresponding result is proved in \cite[Proposition 7.3]{ChaiOort06} -- the situation there is quite different due to the existence of the Frobenius morphism on the base.

Using Faltings' isogeny theorem \cite{Faltings83} and the Cebotarev density theorem, it is possible to show that if $E$ and $E'$ are not themselves geometrically isogenous, then -- after replacing $k$ by a finite extension -- the density of such primes $\mathfrak p$ is zero.

The theorem above is an arithmetic counterpart of the Hodge-theoretic density statement we mentioned for the Galois representation $H^1(E_{\overline k}, \Z_\ell)\otimes H^1(E'_{\overline k}, \Z_\ell)$. Since we do not know of an analogue of the Hodge-theoretic argument in this setting, we offer a different heuristic for Theorem \ref{theorem:main}.

Assume for simplicity that $k$ is the field $\Q$ of rational numbers and that $E$ does not have complex multiplication. Then the Sato-Tate conjecture -- now a theorem proved in \cite{ClozelHarrisTaylor08,Taylor08,HarrisShepherBarronTaylor10} in case $E$ has at least one place of multiplicative reduction -- predicts that as $p$ varies among the prime numbers, the traces $T_p$ of the Frobenius at $p$ are roughly equidistributed between $-2\sqrt p$ and $2\sqrt p$. Assume that the same holds for the traces $T'_p$ associated to $E'$, and that $E$ and $E'$ are not geometrically isogenous. Then one might expect that the distributions of the $T_p$ and $T'_p$ are independent, so that the probability that $T_p$ is equal to $T'_p$ is of the order of $\frac{1}{\sqrt p}$. By Tate's isogeny theorem \cite{Tate66}, $T_p$ and $T'_p$ are equal if and only if the reductions of $E$ and $E'$ modulo $p$ are isogenous. Since the sum over all prime numbers $p$ of the $\frac{1}{\sqrt p}$ diverges, it might be expected that there exist infinitely many primes $p$ such that the reduction of $E$ and $E'$ modulo $p$ are isogenous. 

It seems very difficult to turn the heuristic we just described into a proof. While our proof of Theorem \ref{theorem:main} can likely be made effective, the lower bounds on the number of $\mathfrak p$ satisfying the conclusion is very far from the bounds that could be expected from the discussion above. 

The techniques of our paper are much easier than the ones used in the aforementioned proof of the Sato-Tate conjecture. However, we emphasize that Theorem \ref{theorem:main} does not entail any assumption on the base field nor on the existence of places of multiplicative reduction for $E$ or $E'$. 

\bigskip

Theorem \ref{theorem:main} has consequences for the reduction modulo different primes of a single elliptic curve. 

\begin{corollary}\label{corollary:lang-trotter-weak}
Let $k$ be a number field, and let $E$ be an elliptic curve over $k$. Then at least one of the following statements holds :
\begin{enumerate}
\item There exist infinitely many places $\mathfrak p$ of $k$ such that $E$ has supersingular reduction at $\mathfrak p$. 
\item For any imaginary quadratic number field $K$, there exist infinitely many places $\mathfrak p$ of $k$ such that the reduction of $E$ modulo $\mathfrak p$ has complex multiplication by $K$. 
\end{enumerate}
\end{corollary}

Recall that an elliptic curve over a finite field either has complex multiplication by a quadratic imaginary field or is supersingular. A folklore expectation, related to the Lang-Trotter conjecture \cite{LangTrotter76}, is that both statements of the corollary above should hold unless $E$ has complex multiplication. The existence of infinitely many supersingular primes has been addressed by Elkies \cite{Elkies89}, who managed to prove that statement (1) is always true when $k$ admits a real place. Very little seems to be known about (2) or the general case of (1).

\bigskip

Our proof of Theorem \ref{theorem:main} relies on the Arakelov geometry of the moduli space of elliptic curves. The basic strategy is very simple: given a positive integer $N$, the set of finite places $\mathfrak p$ of $K$ such that the reductions $E_{\mathfrak p}$ and $E'_{\mathfrak p}$ of $E$ and $E'$ modulo $\mathfrak p$ are related -- after some base field extension -- by a cyclic isogeny of degree $N$ can be expressed as the image in $\mathrm{Spec}\,\Z$ of the intersection of an arithmetic curve in $\P^1_{\Z}\times \P^1_{\Z}$ with the graph of a Hecke correspondence $T_N$. We need to show that the union of these places over all possible $N$ is infinite. 

Instead of the set-theoretic intersection, we can consider the intersection number in the sense of Arakelov geometry. Knowing the height of the modular curves by work of Cohen \cite{Cohen84} and Autissier \cite{Autissier03} makes it possible to show that this intersection number grows like $N\log N$ -- assuming $N$ has few prime factors for simplicity. This reduces the proof of the theorem to bounding the local intersection numbers at all places of $k$ -- finite or infinite. This turns out to be, in various forms, a manifestation of the ergodicity of Hecke correspondences as proved in \cite{ClozelOhUllmo01}, though it does not seem to follow directly from it.

As this sketch might suggest, our method of proof is related to the techniques of Gross and Zagier in their celebrated result \cite{GrossZagier86}. Instead of computing intersections of Hecke orbits for Heegner points, we are giving estimates for similar intersection numbers at arbitrary points of the modular curve. Our task is made much simpler technically by the fact that we do not need to prove exact formulas for intersection numbers on modular curves.

As will be apparent in the paper, our proof of Theorem \ref{theorem:main}generalizes to similar statements regarding to the behavior of Hecke correspondences on Shimura curves.

\bigskip

Section 2 is devoted to setting up the notations and proving some basic -- and certainly well-known -- lemmas. In section 3, we show how Corollary \ref{corollary:lang-trotter-weak} can be deduced from the main theorem and reduce the main theorem to local statements. The last two parts of the paper are devoted to the proof of these local statements. 

\bigskip

\noindent\textbf{Acknowledgements.} This paper has greatly benefited from numerous conversations with Emmanuel Ullmo, whom it is a great pleasure to thank. I would like to thank Ching-Li Chai for pointing out the reference \cite{ChaiOort06}.

\section{Notations and preliminary results}

\subsection{Notations}

Let $X(1)$ be the coarse moduli scheme of generalized elliptic curves. It is a 
smooth arithmetic surface over Spec$\,\Z$. The modular invariant $j$ provides 
an isomorphism 
$$j : X(1)\ra \mathbb P^1_{\Z}.$$

Let $\mathbb H$ be the Poincar\'e half-plane, and let $\overline{\mathbb H}$ be 
the union of $\mathbb H$ with the set of cusps $\Q\cup\{\infty\}$. There is a 
canonical isomorphism between the Riemann surfaces $X(1)_{\C}$ and 
$\overline{\mathbb H}/\Gamma(1)$. We will denote by $\tau$ the standard 
coordinate on $\mathbb H$.

Let $N$ be a positive integer, and let $X_0(N)$ be the Deligne-Rapoport 
compactification of the coarse moduli scheme which parametrizes cyclic 
isogenies of degree $N$ between elliptic curves. It is a normal arithmetic surface over 
Spec$\,\Z$.

The two tautological maps from $X_0(N)$ to $X(1)$ induce a self-correspondence 
$T_N$ of $X(1)$. The correspondence $T_N$ is an integral Cartier divisor in 
$X(1)\times_{\mathrm Spec\, \Z} X(1)$. It is called the Hecke correspondence of 
order $N$. Define $e_N=N\,\Pi_{p|N}(1+\frac{1}{p})$, where $p$ runs over the prime 
divisors of $N$. The Hecke correspondence $T_N$ has bidegree $(e_N, 
e_N)$.

\bigskip

Let $\mathcal M$ be the line bundle of modular forms of weight $12$ on $X(1)$. 
The modular form
$$\Delta(\tau)=(2\pi)^{12}q\Pi_{n\geq 1}(1-q^n),$$
with $q=e^{2i\pi\tau}$ induces a global section of $\mathcal M$. It has a zero 
of order $1$ at the cusp $j^{-1}(\infty)$ and does not vanish anywhere else. As 
a consequence, there is a canonical isomorphism $j^*\mathcal 
O(1)\stackrel{\sim}{\ra} \mathcal M$.

If $t\in\C$, the modular form $(t-j)\Delta$ induces a global section of 
$\mathcal M_{\C}$ that has a zero of order $1$ at $j^{-1}(t)$ and does not 
vanish anywhere else. In general, if $Y$ is an horizontal divisor of relative degree $d$ on $X(1)$ with $Y_\C=\sum_i y_i$, then the modular form $\Pi_i (y_i-j)\Delta$ corresponds to the section of $\mathcal M^{\otimes d}$ induced by $Y$.

The Petersson metric on modular forms induces a hermitian metric $||.||$ on $\mathcal M_{\C}$ 
such that $$||\Delta||(\tau)=|\Delta(\tau)||Im(\tau)|^6.$$ This metric is $L^2_1$-singular along the cusp $j^{-1}(\infty)$. The line 
bundle $\mathcal M$ endowed with the Petersson metric on $\mathcal M_{\C}$ is a 
hermitian line bundle $\overline{\mathcal M}$ on $X(1)$.

\subsection{Intersection theory on modular curves}

We will use generalized Arakelov intersection theory for arithmetic surfaces as in \cite{Bost99}, section 5 of which contains the definitions and notations we are using. We also refer to \cite{Autissier03} for the specific case we are considering. The height function with respect to $\overline{\mathcal M}$  is denoted by $h_{\overline{\mathcal M}}$, and arithmetic degrees are denoted by $\widehat{\mathrm{deg}}$. 

The starting point of the proof is the formula giving the height of Hecke 
correspondences. The following is Theorem 3.2 of \cite{Autissier03}, and was also proved in \cite{Cohen84}.

\begin{theorem}\label{theorem:height-of-Hecke-correspondences}
Let $k$ be a number field with ring of integers $\mathcal O_k$. 
Let $Y$ be a horizontal one-dimensional 
integral subscheme of $X(1)_{\mathcal O_K}$ such that $Y_{\C}$ does not meet $j^{-1}(\infty)$. 
Let $d=[k(Y):k]$. Then, as $N$ goes to infinity, we have
\begin{equation}\label{intersection}
h_{\overline{\mathcal M}}(T_{N*}Y)\sim 6d[k:\Q]e_N\log(N)
\end{equation}
\end{theorem}

The estimate above can be rephrased in terms of intersections of divisors on $X(1)_{\mathcal O_K}$. Let $\overline k$ be an algebraic closure of $k$.
Let $P=\sum_i n_iP_i$ be a zero-cycle on $X(1)_{\mathcal O_K}$, where the $P_i$ are closed points. The arithmetic degree of $P$ is defined as
$$\widehat{\mathrm{deg}}(P)=\sum_i n_i \log N(P_i),$$
where $N(P_i)$ is the cardinality of the residue field of $P_i$. Let $Y$ and $Z$ be two divisors in $X(1)$. The arithmetic degree $\widehat{\mathrm{deg}}(Y.Z)$ is defined as the degree of the intersection cycle $Y.Z$.

Finally, let $\sigma$ be an embedding of $k$ into $\C$. If $Z$ is any horizontal divisor on $X(1)_{\mathcal O_K}$, write $Z_{\C}=\sum_{i} n_i Q_i$ with $Q_i\in X(1)(\C)$. Assume that $j(Q_i)\neq\infty$ for all $i$, and let $z_i=j(Q_i)\in \C\subset \mathbb P^1(\C)$. Let $d'$ be the sum of the $n_i$. We denote by $s^{\sigma}_{Z}$ the global section of $\mathcal M_\C^{\otimes d'}$ such that 
$$s^{\sigma}_{Z}(\tau)=\Pi_i(z_i-j(\tau))\Delta(\tau).$$
Then the $s^{\sigma}_{Z}$, as $\sigma$ varies through all embeddings of $k$ into $\C$, extend to a section $s_{Z}$ of $\mathcal M^{\otimes d'}$ over $X(1)_{\mathcal O_K}$.

If $Y$ is any horizontal divisor on $X(1)$ with $Y_\C=\sum_j m_j P_j$, write 
$$s^{\sigma}_{Z}(Y)=\Pi_j s^{\sigma}_Y(P_j)^{m_j}.$$

Let $\sigma_1, \ldots, \sigma_{r_1}$ be the real embeddings of $k$, and let $\sigma_{r_1+1}, \overline{\sigma_{r_1+1}}, \ldots, \sigma_{r_1+r_2}, \overline{\sigma_{r_1+r_2}}$ be the complex embeddings of $k$. We extend the $\sigma_i$ to embeddings $\overline k\ra \C$. As usual, set $\varepsilon_i=1$ if $1\leq i\leq r_1$ and $\varepsilon_i=2$ if $r_1+1\leq i\leq r_1+r_2$. Using the sections $s^{\sigma}_{Z}$ to compute heights with respect to $\mathcal M$, Theorem \ref{theorem:height-of-Hecke-correspondences} gives the following.

\begin{corollary}\label{corollary:global-estimate}
Let $Y$ and $Z$ be two horizontal divisors on $X(1)_{\mathcal O_K}$ of relative degree $d$ and $d'$ respectively. Assume that $Y$ is effective and irreducible, and that for any positive integer $N$, the divisors $T_{N*}Y$ and $Z$ do not have any common component. Then, as $N$ goes to infinity, we have 
$$\widehat{\mathrm{deg}}(Z.T_{N*}Y)-\sum_{i=1}^{r_1+r_2}\varepsilon_i\log ||s^{\sigma_i}_{Z}(T_{N*}Y)||\sim 6dd'[k:\Q]e_N\log(N).$$
\end{corollary}

\section{Local statements and proof of the main results}

In this section, we reduce the proofs of Theorem \ref{theorem:main} and Corollary \ref{corollary:lang-trotter-weak} to local statements which provide estimates for the terms appearing in Corollary \ref{corollary:global-estimate}. The proof of these local estimates will be the core of the paper. 

Let us briefly explain the motivation for the estimates below. We will prove Theorem \ref{theorem:main} by showing that, for suitable large $N$, each individual local term in Corollary \ref{corollary:global-estimate} is negligible before the right-hand-side, which is of the order of $e_N\log(N)$. Let us consider the archimedean term as an example. Let $z$ and $y$ be two complex numbers, that we consider as complex points of $X(1)$. If $N$ is an integer, write the \emph{Hecke orbit} of $y$, $T_{N*}y$, as 
$$T_{N*}y=j(\tau_1)+\ldots+j(\tau_{e_N})$$
where the $\tau_i$ belong to the Poincar\'e upper half-plane $\mathbb H$. We are interested in comparing the quantity
\begin{equation}\label{equation:basic-local-term}
\sum_{i=1}^{e_N} \log(|z-j(\tau_i)|\,||\Delta(\tau_i)||)
\end{equation}
to $e_N\log(N)$ as $N$ goes to infinity.

Equidistribution of Hecke points as proved in \cite{ClozelOhUllmo01} suggest that the sum above should be compared to the integral 
$$e_N\int_{\Gamma\backslash\mathbb H} \log(|z-j(\tau)|\,|\Delta(\tau)|\,|\mathrm{Im}(\tau)|^6)\frac{dx\,dy}{y^2}$$
with $\tau=x+iy$ and $\Gamma=\mathrm{SL}_2(\Z)$. It is readily checked that the latter integral converges, which suggests that the sum (\ref{equation:basic-local-term}) is negligible behind $e_N\log(N)$.

The argument above is not correct, as the function we are integrating has a singularity at the point $j^{-1}(z)$ -- this certainly prevents the use of ergodicity results without modification. However, what the computation above shows is that the estimate we are interested amounts to controlling the best approximations of $z$ by points in the Hecke orbit of $y$.

It might be possible to extend the estimates of our paper so as to study the convergence of $(\ref{equation:basic-local-term})$, but we will be content with weaker estimates.

\bigskip

The three propositions below contain the estimates that allow the argument above to go through. 
The first one deals with the places of bad reduction. From now on, we identify $X(1)$ and $\mathbb P^1$ via the $j$-invariant -- in particular, we see $k$ as a subset of $X(1)(k)$ for any field $k$. 

\begin{proposition}\label{proposition:badred}
Let $|.|$ be a non-archimedian absolute value on $\overline \Q$, and let $K$ be the completion of $\overline\Q$ with respect to $|.|$.
Let $y$ and $z$ be two distinct points of $K\subset X(1)(K)$. 

Assume that the valuation of $y$ is negative. Then there exists a positive 
constant $C$ and an integer $n$ such that for for any positive integer $N$ such 
that $N$ is not a perfect square and $N$ is prime to $n$, the following 
inequality holds : 
\begin{equation}\label{bad-reduction}
\forall \alpha\in T_{N*}y, |z^{-1}-\alpha^{-1}|\geq C.
\end{equation}
\end{proposition}

The following result should be seen as a very weak equidistribution result for Hecke correspondences. For archimedean valuations, it follows from \cite{ClozelOhUllmo01}. For non-archimedean valuations and supersingular reduction, equidistribution has been proved by Fargues (unpublished). For lack of reference, we provide a self-contained argument for our easier result.

\begin{proposition}\label{proposition:equidistribution}
Let $|.|$ be an absolute value on $\overline \Q$, and let $K$ 
be the completion of $\overline \Q$ with respect to $|.|$. Let $y$ and $z$ be 
two distinct points of $K\subset X(1)(K)$. If the absolute value $|.|$ is not archimedean, assume furthermore that the valuation of $y$ is nonnegative. Let $\varepsilon_1$ and $\varepsilon_2$ be two positive real numbers.

There exists a positive constant $\eta$ such that, letting
$$B_{y,z}=\{N\in\mathbb N\setminus (p\mathbb N), |\{\alpha\in T_{N*}y, |z-\alpha|\leq \eta\}|\geq \varepsilon_1 e_N\},$$
where elements in the Hecke orbit of $y$ are counted with multiplicity, then the upper density of $B_{y,z}$ is at most $\varepsilon_2$, i.e.
$$\limsup_{n\ra\infty}\frac{1}{n}|\{N\in B_{y,z} | N\leq n\}\leq \varepsilon_2.$$
\end{proposition}

The last result goes beyond equidistribution. It shows that there cannot exist too many Hecke orbits that contain very good approximations of a given point.

\begin{proposition}\label{proposition:worstapprox}
Let $|.|$ be an absolute value on $\overline \Q$, and let $K$ 
be the completion of $\overline \Q$ with respect to $|.|$. Let $y$ and $z$ be 
two distinct points of $K\subset X(1)(K)$. Assume that $y$ is not the $j$-invariant of a $CM$ elliptic curve and, if the absolute value $|.|$ is not archimedean, assume furthermore that the valuation of $y$ is nonnegative. Let $D$ be a positive integer. 

If the absolute value $|.|$ is not archimedean, let $p$ be the residual characteristic of $k$. We say that an integer $N$ satisfies condition $(P)$ if, if $p$ is odd, $N$ is not congruent to a perfect square modulo $p$ and, if $p=2$, $N$ is congruent to $3$ modulo $4$. In that case, define 
$$S_{y, z}=\{N\in\mathbb N | N\, \mathrm{satisfies\,condition}\,(P)\,\mathrm{and}\,\exists\alpha\in T_{N*}y, |\alpha-z|\leq N^{-D}\}$$
and
$$T_{y, z}=\{N\in\mathbb N\setminus(p\mathbb N) | N\, \mathrm{does\,not\,satisfy\,condition}\, (P)\,\mathrm{and}\,\exists\alpha\in T_{N*}y, |\alpha-z|\leq N^{-D}\}.$$

If the absolute value $|.|$ is archimedean, define
$$S_{y,z}=T_{y,z}=\{N\in\mathbb N | \,\exists\alpha\in T_{N*}y, |\alpha-z|\leq N^{-D}\}.$$

Then there exists a positive integer $D_0$, independent of $y$ and $z$, such that for any $D\geq D_0$, we have 
$$\lim_{n\ra\infty} \frac{1}{n}\min(|S_{y,z}\cap\{1, \ldots, n\}|, |T_{y,z}\cap\{1, \ldots, n\}|)=0.$$
\end{proposition}

The three propositions above will be proved in the next section. We now explain how they imply the main results of the paper.

\begin{proof}[Proof of Theorem \ref{theorem:main}]
We argue by contradiction, and assume that there are only finitely many places $\mathfrak p$ of $k$ as in Theorem \ref{theorem:main}. In particular, $E$ and $E'$ are not geometrically isogenous. We can and will assume that $E$ does not have complex multiplication -- if both curves have complex multiplication, the theorem is clear as there exist infinitely many primes of supersingular reduction for both $E$ and $E'$.

Let $S$ be a finite set of finite places of $k$ containing the places $\mathfrak p$ such that the reductions of $E$ and $E'$ modulo $\mathfrak p$ are smooth and geometrically isogenous as well as the places of bad reduction for $E$ or $E'$. Up to enlarging $k$, we can assume that the only places of bad reduction for $E$ or $E'$ are the places of multiplicative reduction.

Let $y$ (resp. $z$) be the rational point of $X(1)_k$ corresponding to $E$ (resp. $E'$). Let $Y$ and $Z$ be the Zariski closures of $y$ and $z$ in $X(1)_{\mathcal O_k}$ respectively. Let $N$ be a positive integer. Since $E$ and $E'$ are not geometrically isogenous, $Z$ and $T_{N*}Y$ have no common component. By definition of $T_N$, the geometric intersection points of $Z$ and $T_{N*}Y$ correspond precisely to the pairs consisting of a finite place $\mathfrak p$ of $k$ and a cyclic isogeny of degree $N$ between $E_{\overline\kappa(\mathfrak p)}$ and $E'_{\overline\kappa(\mathfrak p)}$, where $\overline\kappa(\mathfrak p)$ is an algebraic closure of the residue field of $\mathfrak p$.

We use the notations of Corollary \ref{corollary:global-estimate} and get, as $N$ goes to infinity
$$\widehat{\mathrm{deg}}(Z.T_{N*}Y)-\sum_{i=1}^{r_1+r_2}\varepsilon_i\log ||s^{\sigma_i}_{Z}(T_{N*}Y)||\sim Ke_N\log(N)$$
for some positive constant $K$.

Write $Z.T_{N*}Y=\sum_i n_iP_i$, where the $P_i$ are closed points of $X(1)_{\mathcal O_k}$. If $\mathfrak p$ is a finite place of $k$, let us write $\mathrm{deg}_\mathfrak{p}(Z.T_{N*}Y)$ for the sum $\sum_i n_i \log N(P_i)$ where $P_i$ runs through the closed points lying over $\mathfrak p$. By definition of $S$, we have $\mathrm{deg}_\mathfrak{p}(Z.T_{N*}Y)=0$ if $\mathfrak p$ does not belong to $S$. As a consequence, we get, as $N$ goes to infinity
\begin{equation}\label{equation:global-estimate}
\sum_{\mathfrak p\in S}\mathrm{deg}_\mathfrak{p}(Z.T_{N*}Y)-\sum_{i=1}^{r_1+r_2}\varepsilon_i\log ||s^{\sigma_i}_{Z}(T_{N*}Y)||\sim Ke_N\log(N).
\end{equation}

Given $\mathfrak p$ in $S$, let $|.|_{\mathfrak p}$ be a non-archimedean absolute value on $\overline \Q$ such that $|p|_\mathfrak{p}=p^{-[k:\Q]}$. If $i$ is an integer between $1$ and $r_1+r_2$, let $|.|_i$ be an archimedean absolute value on $\overline \Q$ extending the absolute value on $k$ defined by the embedding $\sigma_i$ of $k$ in $\C$. For simplicity, we restrict our attention to those integers $N$ which are prime to the residual characteristic of the $\mathfrak p\in S$.

Let $\mathfrak p$ be an element of $S$. For any $\alpha\in T_{N*}y$, the valuation of $y$ with respect to $\mathfrak p$ is negative if and only if the valuation of $\alpha$ is. As a consequence, we have 
$$\mathrm{deg}_\mathfrak{p}(Z.T_{N*}Y)=\sum_{\alpha\in T_{N*}y}\mathrm{Max}(0,-\log|\alpha-z|_\mathfrak{p})$$
if the valuation of $y$ is nonnegative, and 
$$\mathrm{deg}_\mathfrak{p}(Z.T_{N*}Y)=\sum_{\alpha\in T_{N*}y}\mathrm{Max}(0,-\log|\alpha^{-1}-z^{-1}|_\mathfrak{p})$$
if the valuation of $y$ is negative. Here $T_{N*}y$ is seen as a set of $e_N$ distinct $\overline\Q$-points of $X(1)$ -- recall that $E$ does not have complex multiplication.

\bigskip

Let $A$ be the finite set consisting of the absolute values $|.|_{\mathfrak{p}}$ for $\mathfrak p\in S$ and of the archimedean absolute values $|.|_i$. We apply Propositions \ref{proposition:badred}, \ref{proposition:equidistribution} and \ref{proposition:worstapprox} to the absolute values in $A$ simultaneously. Let $\varepsilon$ be a positive real number, that we will take to be small enough.

We first apply Proposition \ref{proposition:badred} simultaneously to those non-archimedean absolute values $|.|_a$ in $A$ for which $|y|_a<1$. We can find a positive constant $C$ and an integer $n$ such that if $N$ is any positive integer which is not a square and is prime to $n$, then 
$$\forall \alpha\in T_{N*}y, |z^{-1}-\alpha^{-1}|_a\geq C$$
for any non-archimedean absolute value $|.|_a$ for which $|y|_a<1$.

We now consider the absolute values $|.|_a$ that are archimedean or satisfy $|y|_a\geq 1$. Applying Proposition \ref{proposition:equidistribution} to those simultaneously, we can find a positive constant $\eta$ and a set of integers $B$ of upper density at most $\varepsilon$ such that for any $|.|_a$ as above, and any integer $N$ prime to the residual characteristic of $|.|_a$, we have
$$N\notin B\implies |\{\alpha\in T_{N*}y, |z-\alpha|\leq \eta\}|\leq \varepsilon e_N.$$

Finally, applying Proposition \ref{proposition:worstapprox}, we can find an integer $D$ such that for any absolute value $|.|_a$ that is archimedean or satisfies $|y|_a\geq 1$, we have 
$$\limsup_{n\ra\infty}\frac{1}{n} |\{1\leq N\leq n|\forall\alpha\in T_{N*}y, \forall a\in A |\alpha-z|\geq N^{-D}\}|\geq \frac{1}{2^{|A|}}.$$

Let $A_{mult}$ be the subset of $A$ consisting of these non-archimedean absolute values $|.|_a$ in $A$ such that $|y|_a<1$. The discussion above shows that we can find infinitely many positive integers $N$ satisfying the following three properties:
\begin{enumerate}[(i)]
\item For any absolute value $|.|_a$ in $A_{mult}$ and any $\alpha$ in $T_{N*y}$, then 
$$|z^{-1}-\alpha^{-1}|_a\geq C;$$
\item for any absolute value $|.|_a$ in $A\setminus A_{mult}$, then 
$$|\{\alpha\in T_{N*}y, |z-\alpha|\leq \eta\}|\leq \varepsilon e_N;$$
\item for any absolute value $|.|_a$ in $A\setminus A_{mult}$ and any $\alpha$ in $T_{N*y}$, then
$$|\alpha-z|\geq N^{-D}.$$
\end{enumerate}

Let us spell out the consequences of these estimates for the intersection numbers we are considering. Let $|.|_a$ be an absolute value in $A$. If $|.|_a$ is in $A_{mult}$, then we can write
\begin{equation}\label{equation:estimate-badred}
\mathrm{deg}_\mathfrak{p}(Z.T_{N*}Y)=\sum_{\alpha\in T_{N*y}}\mathrm{Max}(0,-\log|\alpha^{-1}-z^{-1}|_\mathfrak{p})\leq e_N\log(C^{-1}).
\end{equation}
If $|.|_a$ is non-archimedean and does not belong to $A_{mult}$, we have
\begin{equation}\label{equation:estimate-non-arch}
\mathrm{deg}_\mathfrak{p}(Z.T_{N*}Y)=\sum_{\alpha\in T_{N*y}}\mathrm{Max}(0,-\log|\alpha-z|_\mathfrak{p})\leq e_N\log(\eta^{-1})+\varepsilon e_N D\log(N).
\end{equation}

If $|.|_a$ is archimedean, and corresponds to an embedding $\sigma$ of $k$ in $\C$, choose, for each $\alpha$ in the $N$-th Hecke orbit of $y$, an element $\tau_{\alpha}\in \mathbb H$ such that $j(\tau_{\alpha})=\sigma(\alpha)$ 
$$\log ||s^{\sigma}_{Z}(T_{N*}Y)||=\sum_{\alpha\in T_{N*}y} |\alpha-z|_a||\Delta(\tau_{\alpha})||\geq K_{\Delta}\sum_{\alpha\in T_{N*}y} |\alpha-z|_a,$$
where $K_{\Delta}>0$ is the minimal value of $||\Delta(\tau)||$ for $\tau\in \mathbb H$. As a consequence, we find 
\begin{equation}\label{equation:estimate-arch}
\log ||s^{\sigma}_{Z}(T_{N*}Y)||\geq -K_{\Delta}(e_N\log(\eta^{-1})+\varepsilon e_N D\log(N)).
\end{equation}

\bigskip

We now plug in the estimates (\ref{equation:estimate-badred}), (\ref{equation:estimate-non-arch}) and (\ref{equation:estimate-arch}) in the global degree estimate (\ref{equation:global-estimate}) to find -- after dividing by $e_N$, neglecting the constant terms, and noting that the $\varepsilon_i$ are either $1$ or $2$ --
\begin{equation}
K\log(N)=O((|S\setminus A_{mult}|+ 2(r_1+r_2) K_{\Delta})\varepsilon  D\log(N)).
\end{equation}
In other words, we have 
$$K\leq (|S\setminus A_{mult}|+ 2(r_1+r_2) K_{\Delta})\varepsilon  D.$$
Since $\varepsilon$ can be chosen arbitrarily small and $K>0$, this is a contradiction.
\end{proof}

\begin{proof}[Proof of Corollary \ref{corollary:lang-trotter-weak}]
Assume that $E$ has only finitely many places of supersingular reduction. Let $K$ be an imaginary quadratic extension of $\Q$. We need to show that there exist infinitely many places $\mathfrak p$ of $k$ such that the reduction of $E$ modulo $\mathfrak p$ has complex multiplication by $K$. 

Up to enlarging $k$, we can find an elliptic curve $E'$ defined over $k$ with complex multiplication by $K$. If $\mathfrak p$ is any place of $k$, then the reduction of $E'$ modulo $\mathfrak p$ is either a supersingular elliptic curve or has complex multiplication by $K$. Since $E$ has only finitely many supersingular reductions, Theorem \ref{theorem:main} applied to $E$ and $E'$ shows the result.
\end{proof}

\section{Basic estimates}

The goal of this section is to prove Proposition \ref{proposition:badred} and Proposition \ref{proposition:equidistribution}. The -- more involved -- study of good approximations at the places of good reduction will be dealt with in the next section.

\subsection{Multiplicative reduction}

We prove Proposition \ref{proposition:badred}. It is a consequence of the following more precise result, the proof of which is essentially contained in \cite[Proposition 2.1]{Silverman90}.
 
\begin{proposition}\label{proposition:tatecurveisogeny}
Let $v$ be a non-archimedean valuation on $\overline\Q$ and let $K$ be the 
completion of $\overline\Q$ at $v$. Let $x$ be a rational number and write 
$x=\frac{a}{b}$, where $a$ and $b$ are relatively prime integers. Let $E$ be an 
elliptic curve over $K$ and assume that $v(j(E))$ is negative. Write 
$v(j(E))=\frac{c}{d}$, where $c$ and $d$ are relatively prime integers.r Let 
$N$ be a positive integer such that 
\begin{enumerate}
\item $N$ is not a perfect square ;
\item $N$ is prime to $abcd$.
\end{enumerate}
Let $E'$ be the quotient of $E$ by a subgroup of order $N$. Then 
$$v(j(E'))\neq x.$$
\end{proposition}

\begin{proof}
Since the valuation of $j(E)$ is negative, $E$ is isomorphic to a Tate curve. 
Thus
$$E(K)\simeq \frac{K^*}{q^{\Z}}$$
for some $q\in K^*$, with
$$v(q)=-v(j(E)).$$

Let $N$ be any positive integer. Let $\xi$ be a $N$-th root of $q$, and let 
$\omega$ be a primitive $N$-th root of unity. Then the subgroups of $E$ of 
order $N$ are of the form
$$\frac{\omega^{t\Z}(\omega^s\xi^r)^\Z}{q^\Z},$$
where $r, s$ and $t$ are  positive integers such that $rt=N$ and $s<t$. As 
in \cite[Proposition 2.1]{Silverman90}, it is readily seen that if $E'$ is the quotient of $E$ by 
such a group, then the valuation of $j(E')$ is 
$$v(j(E'))=\frac{r}{t}v(j(E)).$$

Now assume that $N$ satisfies the assumptions of the proposition and that 
$v(j(E'))=x.$ This means that 
$$rbc=tad.$$
Since $N=rt$ is prime to $abcd$, this implies $r=t$, which is a contradiction 
since $N$ is not a square by assumption.
\end{proof}

\begin{proof}[Proof of Proposition \ref{proposition:badred}]
Let $v$ be the valuation corresponding to the absolute value $|.|$. Let $E$ be an elliptic curve over $K$ with $j(E)=y$, and let $x$ be the valuation of $z$. With the notations of Proposition \ref{proposition:tatecurveisogeny}, let $n=abcd.$ Then if $N$ is any positive integer prime to $n$ which is not a perfect square, then Proposition \ref{proposition:tatecurveisogeny} shows that any $\alpha$ in $T_{N*}y$ satisfies $v(\alpha)\neq x=v(z)$. As a consequence, we have 
$$|z^{-1}-\alpha^{-1}|\geq |z^{-1}|.$$
This proves the result.
\end{proof}

\subsection{Equidistribution}

The goal of this section is to prove Proposition \ref{proposition:equidistribution}. We use the notations of the proposition. If the absolute value $|.|$ is archimedean, then the result follows from the main theorem of Clozel, Oh and Ullmo \cite{ClozelOhUllmo01}. 

We now assume that $|.|$ is non-archimedean. Let $v$ be an additive valuation corresponding to $|.|$ and let $p$ be the residual characteristic of $v$. 

Since the valuation of $y$ is negative, so is the valuation of any $\alpha$ in a Hecke orbit of $y$. If the valuation of $z$ is negative, this implies $|z-\alpha|=|z|$ and the result follows. As a consequence, we can assume that the valuation of $z$ is nonnegative.

The assumptions on $y$ and $z$ ensure that we can find a complete discrete valuation ring $W$ whose quotient field is a subfield of $K$, and whose residue field is an algebraically closed field of characteristic $p$, as well as two elliptic curves $E$ and $E'$ over $W$ such that $j(E)=y$ and $j(E')=z$.

Let $\pi$ be a uniformizing parameter of $W$. We can normalize the valuation $v$ so that $v(\pi)=1$. If $n$ is a nonnegative integer, let $W_n$ be the ring $W_n=W/\pi^{n}$, and denote by $H_n$ be the group 
$$H_n=\mathrm{Hom}_{W_n}(E, E')$$
that consists of morphisms from the reduction of $E$ modulo $\pi^n$ to that of $E'$.

The restriction maps $H_n\ra H_{n+1}$ are injective for any $n\geq 0$, see e.g. \cite[Theorem 2.1]{Conrad04}. As a consequence, we consider the sequence $(H_n)_{n\geq 0}$ as a decreasing sequence of subgroups of $H_0$. Grothendieck's existence theorem implies that the intersection $H=\bigcap_{n\geq 0} H_n$ is equal to the group $\mathrm{Hom}_W(E, E')$ of morphisms defined over $W$ considered as a subgroup of $H_0$. 

Let $q$ be the natural positive-definite quadratic form on $H_0$ defined by $q(f)=\mathrm{deg}(f)$. The following basic estimate for number of points in lattices is enough to show the equidistribution result that we need.

\begin{lemma}\label{lemma:low-density}
Let $\varepsilon_1$ and $\varepsilon_2$ be two positive real numbers. There exists a positive integer $n$ such that the set of integers
$$B_n=\{N\in\mathbb N, |q^{-1}(N)\cap H_n|\geq \varepsilon_1 N\}$$
has upper density at most $\varepsilon_2$.
\end{lemma}

\begin{proof}

General results on the endomorphisms groups of elliptic curves show that the groups $H_n$ are free modules of rank $1, 2$ or $4$. If the rank of the $H_n$ is at most $2$, basic geometry of lattices shows that, for any fixed $n$, $|q^{-1}(N)\cap H_n|=O(\sqrt{N})$ as $N$ goes to infinity. As a consequence, we can assume that the $H_n$ have rank $4$. 

Since the intersection $H$ of the $H_n$ has rank at most $2$, as it is equal to the group of morphisms between two elliptic curves over a field of characteristic zero, the index of $H_n$ in $H_0$ goes to infinity with $n$. This  in turn implies that the discriminant of the lattice $H_n$ endowed with the quadratic form $q$ goes to infinity with $n$. 

Let $\varepsilon'$ be a positive real number. A standard volume computation shows in turn that if $n$ is chosen large enough, we can assume that the following estimate holds for any $N$ large enough:
\begin{equation}\label{equation:volume-sphere}
|\{h\in H_n | q(h)\leq N\}|\leq \varepsilon' N^2.
\end{equation}
Indeed, the volume of a sphere of radius $\sqrt{N}$ in the standard euclidean space of dimension $4$ is proportional to $N^2$. Fix such an integer $n$, and let $\varepsilon_1$ and $B_n$ be as in the statement of the lemma. We can write, for any integer $N$ large enough, 
\begin{equation}\label{equation:basic-counting}
|\{h\in H_n | q(h)\leq N\}|\geq \sum_{k\in B_n} |\{q^{-1}(k)\cap H_n\}|\geq \sum_{k\in B_n}\varepsilon_1 k\geq \varepsilon_1 \frac{|B_n|(|B_n|+1)}{2}.
\end{equation}
The estimates (\ref{equation:volume-sphere}) and (\ref{equation:basic-counting}) show that we can write
$$2\varepsilon'N^2\geq \varepsilon_1 |B_n|^2.$$

Choosing $\varepsilon'$ small enough so that $2\varepsilon'\leq \varepsilon_1\varepsilon^2_2$, the estimate above shows that, for suitable $n$ and for all $N$ large enough, we have $|B_n|\leq \varepsilon_2 N$. This proves the result.

\end{proof}

\begin{proof}[Proof of Proposition \ref{proposition:equidistribution}]
We keep the notations above, and choose $n$ as in Lemma \ref{lemma:low-density}. Let $N$ be an integer prime to $p$. Then the group scheme $E[N]$ defined as the kernel of multiplication by $N$ is \'etale over $W$ since $N$ is prime to $p$. Since the residue field of $W$ is algebraically closed by assumption, $E[N]$ is isomorphic to the constant group $(\Z/N\Z)^2$. In other words, the $N$-torsion points of $E$ are defined over $W$.

Let $\alpha$ be a point in the Hecke orbit $T_{N*}y$ of $y$. The remark above shows that there exists a unique elliptic curve $E_{\alpha}$ over $W$, together with a cyclic isogeny $E\ra E_\alpha$ of degree $N$. 

Proposition 2.3 of \cite{GrossZagier85} states the we have the following equality
$$v(\alpha-z)=\sum_{n\geq 1}\frac{|\mathrm{Iso}_n(E_\alpha, E')|}{2}, $$
where $\mathrm{Iso}_n(E_\alpha, E_\beta)$ denotes the set of isomorphisms from $E_\alpha$ to $E_\beta$ that are defined over $W/\pi^n$. Let $n$ be the largest integer such that $E_\alpha$ and $E'$ are isomorphic over $W/\pi^n$. Since the group of automorphisms of an elliptic curve over $W/\pi^n$ has cardinality at most $24$, we have 
$$n\geq \frac{1}{12}v(\alpha-z).$$

The curve $E_{\alpha}$ is isomorphic to $E'$ modulo $\pi^n$. The composition 
$$E_n\ra E_{\alpha, n}\ra E'_n \,\,\,\,\, (\mathrm{modulo}\, \pi^n)$$
is an element $h_{\alpha}\in q^{-1}(N)\cap H_n$. The isogeny $h_{\alpha}$ is well-defined up to automorphisms of $E'_n$, and $h_{\alpha}$ determines $\alpha$ up to automorphism. Since the automorphism group of an elliptic curve is at most of order $24$, the discussion above shows that for any positive integer $n$, we have 
$$|\{\alpha\in T_{N*}y, v(\alpha-z)\geq 12n\}|\leq 24 |q^{-1}(N)\cap H_n|.$$

Let $\eta$ be a positive real number such that $|z-\alpha|\leq\eta$ implies $v(z-\alpha)\geq 12n$ for any $\alpha$ in $W$. Since for any positive $N$, $e_N\geq N$, Lemma \ref{lemma:low-density} shows that the set 
$$B_{y,z}=\{N\in\mathbb N\setminus (p\mathbb N), |\{\alpha\in T_{N*}y, |z-\alpha|\leq \eta\}|\geq \varepsilon_1 e_N\},$$
has upper density at most $\varepsilon_2$. This shows the result.
\end{proof}

\section{Bounding the best approximations}

The goal of this section is to prove Proposition \ref{proposition:worstapprox}. We handle the case of non-archimedean valuations and archimedean valuations separately. The non-archimedean case is proved in Proposition \ref{proposition:bad-approximation-p-adic}, and the archimedean case -- which reduces to the study of the usual absolute value on $\C$ -- is Proposition \ref{proposition:bad-approximation-archimedean}.

Before getting to the actual proofs, we record some elementary results. If $L$ is a lattice, let $\mathrm{disc}(L)$ denote its discriminant, which is well-defined as an element of $\Z$.

\begin{lemma}\label{lemma:bounding-discriminant}
Let $E$ and $E'$ be two isogenous CM elliptic curves over an algebraically closed field $K$. Consider $\mathrm{End}(E)$ and $\mathrm{Hom}(E, E')$ as lattices with respect to the quadratic form $q(f)=\mathrm{deg}(f)$. Then we have 
$$|\mathrm{disc}(\mathrm{End}(E))|\geq |\mathrm{disc}(\mathrm{Hom}(E, E'))|.$$
\end{lemma}

\begin{proof}
Let $f: E'\ra E$ be a cyclic isogeny, and consider the morphism
$$\alpha : \mathrm{Hom}(E, E')\ra \mathrm{End}(E), \phi\mapsto f\circ\phi.$$
Then for any $\phi\in \mathrm{Hom}(E, E')$, we have $q(\alpha(\phi))=q(f)q(\phi)$. In particular, we have
$$|\mathrm{disc}(\alpha(\mathrm{Hom}(E, E')))|=q(f)^2|\mathrm{disc}(\mathrm{Hom}(E, E'))|.$$
Now the cokernel of $\alpha$ has length at least $q(f)$. Indeed, if $n$ is a positive integer, let $[n]$ denote multiplication by $n$ on $E$. Then the kernel of $[n]$ is isomorphic to $(\Z/n\Z)^2$ as an abelian group. For $n$ to belong to the image of $\alpha$, the group $(\Z/n\Z)^2$ needs to contain an element of order $q(f)$, so $q(f)$ must divide $n$. 

This shows that 
$$|\mathrm{disc}(\alpha(\mathrm{Hom}(E, E')))|=|\mathrm{disc}(\mathrm{End}(E))||\mathrm{Coker}(\alpha)|^2\leq |\mathrm{disc}(\mathrm{End}(E))|q(f)^2,$$
which shows the result.
\end{proof}

We record a proof of the following elementary lemma without any regards for the optimality of the constants involved.

\begin{lemma}
Let $L$ be a rank $2$ lattice with positive-definite quadratic form $q$ and discriminant $\delta$. Then for any positive integer $n$, we have
$$|\{N\leq n\, |\, \exists l\in L, \, q(l)=N\}|\leq 1+8\sqrt{n}+\frac{16n}{\sqrt{\delta}}$$
\end{lemma}

\begin{proof}
Lagrange reduction for rank $2$ lattices shows that we can find a basis $(e, f)$ for $L$ such that $q(e)\leq q(f)$ and $2|q(e, f)|\leq q(e)$. Let $a$ and $b$ two integers. Then 
$$q(ae+bf)=a^2q(e)+b^2q(f)+2abq(e, f)\geq \frac{1}{2}(a^2q(e)+b^2q(f)).$$
In particular, $q(ae+bf)\leq n$ implies $a^2q(e)\leq 2n$ and $b^2q(f)\leq 2n$. This implies in particular that the set $\{(a, b)|q(ae+bf)\leq n\}$ has cardinality at most
$$|\{(a, b)|q(ae+bf)\leq n\}|\leq (1+\frac{4\sqrt{n}}{\sqrt{q(e)}})(1+\frac{4\sqrt{n}}{\sqrt{q(f)}})\leq 1+8\sqrt{n}+\frac{16n}{\sqrt{q(e)q(f)}}.$$
The discriminant $\delta$ of $L$ is
$$\delta=q(e)q(f)-q(e, f)^2\leq q(e)q(f).$$
Since of course we have
$$|\{N\leq n|\, \exists l\in L, \, q(l)=N\}|\leq |\{(a, b)|q(ae+bf)\leq n\}|,$$
this finishes the proof.
\end{proof}

Putting the two lemmas above together, we find the following statement.

\begin{proposition}\label{proposition:low-density}
Let $E$ and $E'$ be two CM elliptic curves over an algebraically closed field $K$. Let $\delta$ be the discriminant of the lattice $\mathrm{End}(E)$. Then, for any positive integer $n$, we have 
$$|\{N\leq n |\, \exists \phi\in\mathrm{Hom}(E, E'),\,\mathrm{deg}(\phi)=N\}\leq 1+8\sqrt{n}+\frac{16n}{\sqrt{\delta}}.$$
\end{proposition}

\subsection{The non-archimedean case}

In this section, let $|.|$ be a non-archimedean absolute value on $\overline\Q$, and let $K$ be the completion of $\overline\Q$ with respect to $|.|$. 

The goal of this section is to show that, given two points $y$ and $z$ of $X(1)$, there exist sufficiently many Hecke orbits of $y$ that do not contain elements that are too close to $z$ with respect to a given absolute value -- with precise estimates to be given below.  Our argument is the following: we will show that the only way two different Hecke orbits of $y$ can both contain very good approximations of $z$ is if $y$ is very close to a CM point whose ring of endomorphisms we can control. Since distinct CM points of low discriminant cannot be too close to one another, this will imply that some Hecke orbits of $y$ can't contain very good approximations of $z$.

In order to facilitate the exposition, we will give slightly different statements in the archimedean and in the non-archimedean situations. Nevertheless, it would be possible to give completely uniform estimates. 

We start with the easier case where $|.|$ is supposed to be non-archimedean, and let $v$ be an additive valuation associated to $|.|$. Let $p$ be the residual characteristic of $v$. The lifting results we will need take a most simple form when introducing the following condition.

\begin{definition}\label{definition:ring-of-integers}
Let $N$ be a positive integer. If $p$ is odd, we say that $N$ satisfies condition $(P)$ if $N$ is not congruent to a perfect square modulo $p$. If $p=2$, we say that $N$ satisfies condition $(P)$ if $N$ is congruent to $3$ modulo $4$. 
\end{definition}

The reason we introduced the condition above is expressed in the following lemma.

\begin{lemma}\label{lemma:p-implies-integrally-closed}
Let $N$ be a positive integer. If $N$ satisfies condition $(P)$, then $N$ is prime to $p$ and, for any totally imaginary quadratic integer $\alpha$ of norm $N$, the index of the ring $\Z[\alpha]$ in its integral closure in $\Q[\alpha]$ is prime to $p$.
\end{lemma}

\begin{proof}
$N$ is prime to $p$ by definition. Let $\alpha$ be a purely imaginary quadratic integer of norm $N$. Then we can find an integer $t$ with $t^2-4N<0$ such that 
$$\alpha^2-t\alpha+N=0.$$
Let $d$ be a squarefree negative integer with $\Q[\sqrt{\alpha}]=\Q[\sqrt{d}]$, and write 
$$\alpha=\lambda+\mu \sqrt{d}$$
for some rational numbers $\lambda$ and $\mu$. Then $2\lambda=t$ and $\lambda^2-d\mu^2=N$. Let us first assume that $p$ is odd. Since $N$ is not a square modulo $p$, the $p$-adic valuation of $\mu$ cannot be positive, which shows the result.

Now assume that $p=2$ and that $N$ is congruent to $3$ modulo $4$. It is readily seen that $\mu$ cannot be chosen to have positive dyadic valuation, which shows the result if $d$ is congruent to $2$ or $3$ modulo $4$. If $d$ is congruent to $1$ modulo $4$, and if $\lambda$ and $\mu$ both have nonnegative dyadic valuation, then $\lambda^2+d\mu^2$ is congruent to $0, 1$ or $2$ modulo $4$, and as a consequence cannot be equal to $N$, which shows the result.
\end{proof}

\begin{lemma}\label{lemma:finding-CM-p-adic-1}
Let $W\subset \mathcal O_K$ be a discrete valuation ring with algebraically closed residue field, and let $E$ be an elliptic curve over $W$. Let $\pi$ be a prime element of $W$. For any nonnegative integer $n$, we denote by $E_n$ the reduction of $E_n$ modulo $n$. 

Let $N$ be a positive integer that satisfies condition $(P)$ of definition \ref{definition:ring-of-integers}, and, for some $n\geq 4\frac{v(p)}{v(\pi)}$, let $\phi_n : E_n\ra E_n$ be a cyclic isogeny of degree $N$. Then the pair $(E_{n-4\frac{v(p)}{v(\pi)}}, \phi_{n-4\frac{v(p)}{v(\pi)}})$ lifts uniquely to a pair $(E_{CM}, \phi_{CM})$ over $W$.
\end{lemma}

\begin{proof}
We can assume that $v(\pi)=1$. On the tangent space $\mathrm{Lie}(E_n)$, $\phi_n$ induces multiplication by an element $w_n$ in $W/\pi^n$. As an application of Lubin-Tate theory, it is proved in \cite[Proposition 2.7]{GrossZagier85} that condition $(P)$ guarantees\footnote{The assumption in \cite{GrossZagier85} is actually that $\Z[\phi_n]$ is integrally closed in its fraction field. However, one only needs for the proof to go through is for the index of $\Z[\phi_n]$ in its integral closure to be prime to $p$.} that the pair $(E_n, \phi_n)$ lifts uniquely to a pair $(E_{CM}, \phi_{CM})$, where $E_{CM}$ is an elliptic curve over $W$ and $\phi_{CM} : E_{CM}\ra E_{CM}$ is a cyclic isogeny of degree $N$, as soon as the equation 
$$X^2-\mathrm{Tr}(\phi_n)X+N=0$$
has a solution $w$ in $W$ that is congruent to $w_n$ modulo $\pi^n$.

We chose $N$ in such a way that the discriminant $\mathrm{Tr}(\phi_n)^2-4N$ is not divisible by $p^4$ (in case $p$ is odd, it is actually not divisible by $p$). Then Hensel's lemma shows that there exists a solution $w$ in $W$ of the equation above that is congruent to $w_n$ modulo $\pi^{n-4v(p)}$. The discussion above shows that the pair $(E_{n-4v(p)}, \phi_{n-4v(p)})$ lifts uniquely to a pair $(E_{CM}, \phi_{CM})$ over $W$. 
\end{proof}

\begin{proposition}\label{proposition:finding-CM-p-adic}
Let $y$ be an element of $K\subset X(1)(K)$ with nonnegative valuation. Then for any two distinct positive integers $N_1, N_2$ such that $N_1N_2$ satisfies condition $(P)$, there exists a $CM$ elliptic curves $E_{CM}$ with a cyclic self-isogeny of degree at most $N_1N_2$ such that for any $(\alpha, \beta)\in  (T_{N_1*}y)\times (T_{N_2*}y)$, we have
$$|y-j(E_{CM})| \leq |p|^{-4} |\alpha-\beta|^{1/12}.$$
\end{proposition}

\begin{proof}
Let $W\subset \mathcal O_K$ be a discrete valuation ring with algebraically closed residue field, and let $E$ be an elliptic curve over $W$ with $j$-invariant $y$.

Fix $\alpha$ and $\beta$ in $T_{N_1*}y$ and $T_{N_2*}y$ respectively. There exist unique elliptic curves over $K$ with $j$-invariant $\alpha$ and $\beta$ respectively. Since $N_1$ and $N_2$ are both prime to $p$, these curves extend to elliptic curves $E_{\alpha}$ and $E_{\beta}$ over $W$, together with cyclic isogenies $\phi : E\ra E_{\alpha}$ and $\psi : E\ra E_\beta$ of degree $N_1$ and $N_2$ respectively. We have $E_\alpha=E/\Ker(\phi)$ and $E_\beta=E/\Ker(\psi)$.

Let $\pi$ be a prime element of $W$. We can assume that $v(\pi)$=1. Proposition 2.3 of \cite{GrossZagier85} states the we have the following equality
$$v(\alpha-\beta)=\sum_{n\geq 1}\frac{|\mathrm{Iso}_n(E_\alpha, E_\beta)|}{2}, $$
where $\mathrm{Iso}_n(E_\alpha, E_\beta)$ denotes the set of isomorphisms from $E_\alpha$ to $E_\beta$ that are defined over $W/\pi^n$. Let $n$ be the largest integer such that $E_\alpha$ and $E_{\beta}$ are isomorphic over $W/\pi^n$. Since the group of automorphisms of an elliptic curve over $W/\pi^n$ has cardinality at most $24$, we have 
$$n\geq \frac{1}{12}v(\alpha-\beta).$$

If $n\leq 4v(p)$, then $|p|^{-4} |\alpha-\beta|^{1/12}\geq 1$ and we can choose $E_{CM}$ arbitrarily. We assume that $n\geq 4v(p)$.

\bigskip

The morphism $\phi$ induces a cyclic isogeny of degree $N_1$ 
$$\hat{\phi} : E_\alpha\ra E.$$
Since $E_\alpha$ and $E_\beta$ are isomorphic modulo $\pi^n$, $\hat{\phi}$ induces a cyclic isogeny of degree $N_1$ from the reduction of $E_\beta$ modulo $\pi^n$ to the reduction $E_n$ of $E$ modulo $\pi^n$. Working over $W/\pi^n$, we can write $\hat{\phi}\circ\psi=[x]\circ f$, where $x$ is a positive integer and $f$ is a cyclic self-isogeny of $E_n$. Since $N_1N_2$ satisfies condition $(P)$, so does the degree of $f$. By Lemma \ref{lemma:finding-CM-p-adic-1}, the pair $(E_{n-4v(p)}, f)$ lifts uniquely to a pair $(E_{CM}, f)$ over $W$. We have
$$v(y-j(E_{CM}))\geq n-4v(p)\geq \frac{1}{12}v(\alpha-\beta)-4v(p),$$
which shows the result.
\end{proof}

The preceding corollary will be used in conjunction with the following result, which asserts that CM points cannot be too close.

\begin{proposition}\label{proposition:CM-far-apart-p-adic}
Let $M_1$ and $M_2$ be two positive integers both relatively prime to $p$. Let $E_1$ and $E_2$ be two elliptic curves over $K$ with cyclic self-isogenies of degree $M_1$ and $M_2$ respectively. If $E_1$ and $E_2$ are not isomorphic, then 
$$|j(E_1)-j(E_2)|\geq (4M_1M_2)^{-C} ,$$
where $C$ is a positive constant depending only on the absolute value $|.|$.
\end{proposition}

\begin{proof}
As before, let us choose a discrete valuation ring $W\subset \mathcal O_K$ with algebraically closed residue field such that $E_1$ and $E_2$ are both defined over $W$. Since $M_1$ and $M_2$ are prime to $p$, we can find cyclic isogenies $\phi_i : E_i\ra E_i$ of degree $M_i$ for $i=1,2$.

Let $n$ be the largest integer such that $E_1$ and $E_2$ are isomorphic modulo $W/\pi^n$. We can assume that $n\geq 1$, i.e., that the reductions of $E_1$ and $E_2$ modulo $\pi$ are isomorphic over the field $k=W/\pi$ to the elliptic curve $E$. As above, we have
\begin{equation}\label{equation:control-n}
n\geq \frac{1}{12}v(j(E_1)-j(E_2)).
\end{equation}

If $E$ is a CM elliptic curve, then $E_1$ and $E_2$ both are isomorphic over $K$ to the canonical lifting of $E$, which is impossible since we assumed that $E_1$ and $E_2$ are not isomorphic. As a consequence, we can assume that $E$ is a supersingular elliptic curve, and fix an isomorphism from the reduction of $E_1$ over $W/\pi^n$ to that of $E_2$.

Let $\Gamma_1$ and $\Gamma_2$ be the formal $\Z_p$-modules $E_1[p^{\infty}]$ and $E_2[p^{\infty}]$ over $W$. If $m$ is any positive integer, let $\mathrm{End}_m(\Gamma_i)$ be the group of endomorphisms of $\Gamma_i$ over $W/\pi^m$ for $i=1, 2$. Since $E_1$ and $E_2$ are isomorphic over $W/\pi^n$, we have.
$$\mathrm{End}_n(\Gamma_1)=\mathrm{End}_n(\Gamma_2).$$

Now let $\mathcal O_i=\mathrm{End}_W(\Gamma_i)$, and let $D=\mathrm{End}_1(\Gamma_1)=\mathrm{End}_1(\Gamma_2)$. As subgroups of $D$, we have $\mathcal O_1\neq \mathcal O_2$, for instance by \cite[Proposition 2.1]{Gross86}. We $\mathcal O_i[\frac{1}{p}]=\Q_p[\phi_i]$ for $i=1,2$.

In \cite{Gross86}, Gross computes the endomorphisms groups above and shows the equality, for any positive integer $m$,
$$\mathrm{End}_m(\Gamma_i)=\mathcal O_i+p^mD.$$
As a consequence, we have, as subgroups of $D$, 
$$\mathcal O_1 + p^nD=\mathcal O_2+p^nD$$
and in particular
$$\mathcal O_1/p^n\mathcal O_1=\mathcal O_2/p^n\mathcal O_2$$
as subgroups of $D/p^nD$.

Since $\mathcal O_1$ and $\mathcal O_2$ are commutative algebras, the formula above shows that the commutator $[\phi_1, \phi_2]$ belongs to $p^nD$. However, since $\phi_2$ does not belong to $\mathcal O_1$, it does not belong to $\mathcal O_1[\frac{1}{p}]$ since $\phi_2$ is integral over $\Z_p$ and $\mathcal O_1$ is integrally closed in its fraction field. Since $\mathcal O_1[\frac{1}{p}]$ is its own commutant in the quaternion algebra $D[\frac{1}{p}]$, this means that the commutator $[\phi_1, \phi_2]$ is not zero. In particular, its norm is at least $p^{2n}$. 

Now by assumption the norm of $\phi_1$ (resp. $\phi_2$) in $D$ is $M_1$ (resp. $M_2$). As a consequence, the norm of $[\phi_1, \phi_2]$ is at most $4M_1M_2$. Finally, we get
$$4M_1M_2\geq p^n.$$
Together with (\ref{equation:control-n}), this shows the result.

\end{proof}

Putting Proposition \ref{proposition:finding-CM-p-adic} and Proposition \ref{proposition:CM-far-apart-p-adic} together, we obtain the following.

\begin{proposition}\label{proposition:bad-approximation-p-adic}
Let $D$ be a large enough positive integer. For any $y, z\in K\subset X(1)(K)$, define
$$S_{y, z}=\{N\in\mathbb N | N\, \mathrm{satisfies\,condition}\,(P)\,\mathrm{and}\,\exists\alpha\in T_{N*}y, |\alpha-z|\leq N^{-D}\}$$
and
$$T_{y, z}=\{N\in\mathbb N\setminus(p\mathbb N) | N\, \mathrm{does\,not\,satisfy\,condition}\, (P)\,\mathrm{and}\,\exists\alpha\in T_{N*}y, |\alpha-z|\leq N^{-D}\}.$$
Assume that $y$ has nonnegative valuation and is not the $j$-invariant of a CM elliptic curve. Then 
$$\lim_{n\ra\infty} \frac{1}{n}\min(|S_{y,z}\cap\{1, \ldots, n\}|, |T_{y,z}\cap\{1, \ldots, n\}|)=0.$$
\end{proposition}

\begin{proof}
Let $n$ be a large enough integer. We will determine $D$ independently of $n$ later. Fixing $y$ and $z$ as above, we need to show that at least one of the quantities $\frac{1}{n}|S_{y,z}\cap\{1, \ldots, n\}|$, $\frac{1}{n}|T_{y,z}\cap\{1, \ldots, n\}|$ is very small.

Let $C$ be a positive constant as in Proposition \ref{proposition:CM-far-apart-p-adic}. If one of the sets $S_{y,z}$, $T_{y, z}$ does not contain any integer between $\sqrt{n}$ and $n$, then we are done. As a consequence, we can assume that there exists $\sqrt{n}\leq N_1\leq n$ satisfying condition $(P)$, and $\sqrt{n}\leq N_2\leq n$ prime to $p$ and not satisfying condition $p$, as well as $\alpha\in T_{N_1*}y$, $\beta\in T_{N_2*}y$ such that 
$$|\alpha-z|\leq N_1^{-D}\leq n^{-D/2},\,\, |\beta-z|\leq N_2^{-D}\leq n^{-D/2}.$$
In particular, we have $|\alpha-\beta|\leq 2n^{-D/2}.$ By construction, $N_1N_2$ satisfies condition $(P)$. As a consequence, Proposition \ref{proposition:finding-CM-p-adic} shows that we can find an elliptic curve $E_{CM}$ with a cyclic self-isogeny of degree at most $n^2$ such that 
\begin{equation}\label{proche-de-CM}
|y-j(E_{CM})|\leq Kn^{-D/24}
\end{equation}
for some positive constant $K$ independent of $y, z$ and $n$. Choosing $D$ large enough, Proposition \ref{proposition:CM-far-apart-p-adic} shows that $E_{CM}$ is unique: there exists a unique CM elliptic curve with a cyclic self-isogeny of degree at most $n$ satisfying the estimate above.

Let $N$ be any element of $S_{y,z}$ with $\sqrt{n}\leq N\leq n$. By assumption, there exists an element $\alpha_N$ of $T_{N*}y$ such that $|\alpha_N-z|\leq n^{-D/2}$. The estimate (\ref{proche-de-CM}) readily shows, using arguments as above and up to enlarging $K$, that there exists an element $\alpha'_{CM}$ of $T_{N*}j(E_{CM})$ with 
$$|\alpha'_{CM}-z|\leq Kn^{-D/24}.$$
By construction, $\alpha'_{CM}$ is the $j$-invariant of a $CM$ elliptic curve $E'_{CM}$ with a cyclic isogeny $E_{CM}\ra E'_{CM}$ of degree $N$. In particular, $E'_{CM}$ admits a cyclic self-isogeny of degree at most $n^2$. Choosing again $D$ large enough, Proposition \ref{proposition:CM-far-apart-p-adic} shows that $E'_{CM}$ is uniquely determined by these conditions. 

\bigskip

We can rephrase the discussion of the previous paragraphs as follows: if both $S_{y, z}$ and $T_{y, z}$ contain an element $N$ between $\sqrt{n}$ and $n$ -- still for $n$ large enough -- then there exist unique $CM$ elliptic curves $E_{CM}$ and $E'_{CM}$ with cyclic self-isogenies of degrees at most $n^2$ such that for any integer $N$ between $\sqrt{n}$ and $n$, if $N$ belongs to $S_{y, z}\cup T_{y, z}$, then there exists a morphism of degree $N$ between $E_{CM}$ and $E'_{CM}$. Let $\delta_{CM}$ be the discriminant of $E_{CM}$. Together with Proposition \ref{proposition:low-density}, we find 
\begin{equation}\label{equation:basic-density}
|(S_{y, z}\cup T_{y, z})\cap \{1, \ldots, n\}|\leq 1+8\sqrt{n}+\frac{16n}{\delta_{CM}}+\sqrt{n}=1+9\sqrt{n}+\frac{16n}{\delta_{CM}}.
\end{equation}

Since $y$ is not the $j$-invariant of a $CM$ elliptic curve, the estimate (\ref{proche-de-CM}) implies that $E_{CM}$ takes infinitely many values as $n$ grows. Since there are only finitely many $CM$ elliptic curves over $K$ with bounded discriminant, this shows that $\delta_{CM}$ goes to infinity with $n$. Equation (\ref{equation:basic-density}) allows us to conclude.
\end{proof}

It should be noted that the use of condition $(P)$ is not strictly necessary for the proposition above, though its use made the estimates involved both stronger and easier to prove. We state without proof the following improved version of the density result we just stated -- it is not used anywhere in this paper.

\begin{proposition}
Let $D$ be a large enough positive integer. For any $y, z\in K\subset X(1)(K)$, define
$$S_{y, z}=\{N\in\mathbb N\setminus (p\mathbb N) | \exists\alpha\in T_{N*}y, |\alpha-z|\leq N^{-D}\}.$$
Assume that $y$ has nonnegative valuation and is not the $j$-invariant of a CM elliptic curve. Then $S_{y, z}$ has density zero, i.e. 
$$\lim_{n\ra\infty} \frac{1}{n}|S_{y,z}|=0.$$
\end{proposition}

\subsection{The archimedean case}

We now transpose the results of our preceding question to the archimedean setting, following the same strategy. We work over the field of complex numbers, and let $|.|$ be the usual absolute value on $\C$. 

We start with two elementary lemmas.

\begin{lemma}\label{lemma:bounding-coefficients}
Let $\tau$ be an element of $\mathbb H$. Let $N$ be a positive integer, and let $a, b, c, d$ be integers with $ad-bc=N$. Let 
$$f : \mathbb H\ra \mathbb H, z\mapsto \frac{az+b}{cz+d}.$$
Let $K$ be a compact subset of $\mathbb H$. Then there exists a positive constant $K_0$ depending only on  $K$ such that 
$$(\tau, f(\tau))\in K^2 \implies \mathrm{Max}(|a|, |b|, |c|, |d|)\leq K_0\sqrt{N}.$$
\end{lemma}

\begin{proof}
We keep the notations of the lemma. Then, for any $\tau\in \mathbb H$, we have 
$$\mathrm{Im} f(\tau)=\frac{N\mathrm{Im}(\tau)}{|c\tau+d|^2}.$$
As a consequence, choosing $\varepsilon_0$ small enough, the inequality 
$$\mathrm{Im}(\tau)\geq \inf_{\tau\in K} \mathrm{Im}(\tau)>0$$
implies $|c\tau+d|\leq \lambda\sqrt{N}$ for some constant $\lambda$ depending only on $\tau$ and $\varepsilon_0$. Since the imaginary part of $\tau$ is positive, this implies the required estimate on $c$, and consequently on $d$. Furthermore, since $|c\tau+d|\leq \lambda \sqrt{N}$, we can find a constant $\mu$ depending only on $\tau$ and $\varepsilon$ such that $|a\tau+b|\leq \mu\sqrt{N}$ as well, which provides the required estimates for $a$ and $b$ as well.
\end{proof}

\begin{lemma}\label{lemma:close-to-fixed-point}
Let us keep the notations of the previous lemma and assume $f\neq \mathrm{Id}_\mathbb{H}$. Then there exists positive constants $\varepsilon_1$, $K_1$ depending only on $\tau$, such that 
$$|\tau-f(\tau)|\leq \frac{\varepsilon_1}{N}\implies \exists\,\tau_0\in \mathbb H,\,f(\tau_0)=\tau_0 \, \mathrm{and}\,  |\tau-\tau_0|\leq K_1\sqrt{N}|\tau-f(\tau)|.$$
\end{lemma}

\begin{proof}
Choose $\varepsilon <\varepsilon_0$. We leave it to the reader to show that if $c=0$ and $\varepsilon_1$ is chosen small enough, then there is no $\tau\in\mathbb H$ satisfying 
$$|\tau-f(\tau)|\leq \frac{\varepsilon_1}{N}.$$

We now assume $c\neq 0$. In that case $f$ has two fixed points in $\C$, $\tau_0$ and $\tau'_0$. We can write
$$|\tau-f(\tau)|=\frac{|c||\tau-\tau_0||\tau-\overline{\tau_0}|}{|c\tau+d|}.$$
Lemma \ref{lemma:bounding-coefficients} proves that $|c|$ and $|d|$ are bounded above by $K_0\sqrt{N}$ for some constant $K_0$ depending only on $\tau$, it is readily shown that -- once again choosing $\varepsilon_1$ to be small enough, the existence of $\tau\in \mathbb H$ with 
$$|\tau-f(\tau)|\leq \frac{\varepsilon_1}{N}$$
implies that $\tau_0$ and $\tau'_0$ are complex conjugates and are not real numbers. We can assume that $\tau_0$ is in $\mathbb H$. As a consequence, we have $|\tau-\tau_0|\geq \mathrm{Im}(\tau)$. Putting these estimates together shows the result.
\end{proof}

\begin{proposition}\label{proposition:finding-CM-archimedean-1}
Let $y$ be an element of $\C\subset X(1)(\C)$. Then there exists a positive constants $C, \varepsilon$ such that for any integer $N>1$ and any $\alpha\in T_{N*}y$ such that 
$$|y-\alpha|\leq \frac{\varepsilon}{N},$$
 there exists a $CM$ elliptic curve $E_0$ over $\C$ with a cyclic self-isogeny of degree $N$ such that 
$$|y-j(E_0)|\leq C\sqrt{N}|y-\alpha|.$$
\end{proposition}

\begin{proof}
Let $\tau$ be an element of $\mathbb H$ with $j(\tau)=y$. Let $\varepsilon_2, K_2$ be positive real numbers such that for any $z\in \C\subset X(1)(\C)$ with $|z-y|\leq \varepsilon_2$, we can find $\tau'$ with $j(\tau')=z$ and 
$$K_2^{-1}|y-z|\geq |\tau-\tau'|\geq K_2|y-z|.$$

The preimage of $T_{N*}y$ by $j$ exactly the set of element $\frac{a\tau+b}{c\tau+d}\in \mathbb H,$ with 
\begin{equation}\label{equation:condition-on-homography}
\begin{pmatrix}
a & b\\
c & d
\end{pmatrix}
\in \mathrm{SL_2}(\Z)
\begin{pmatrix}
\alpha & \beta\\
0 & \delta
\end{pmatrix},
\end{equation}
where $\alpha, \beta$ and $\delta$ are three integers with no common factor and $\alpha\delta=N$, $\alpha\geq 1$, $0\leq \beta<\delta$.

Let us consider $\alpha\in T_{N*}y$ with $|y-\alpha|\leq \varepsilon_2$. We can write $\alpha=j(\frac{a\tau+b}{c\tau+d})$ with $a,b,c,d$ as in (\ref{equation:condition-on-homography}) and $|\tau -\frac{a\tau+b}{c\tau+d}|\leq K_2^{-1}|y-\alpha|$. Now choose $\varepsilon_1$ as in Lemma \ref{lemma:close-to-fixed-point} and assume that $|y-\alpha|\leq K_2\frac{\varepsilon}{N}$. We can find $\tau_0\in \mathbb H$ such that 
\begin{equation}\label{equation:tau0-is-CM}
\tau_0=\frac{a\tau_0+b}{c\tau_0+d}
\end{equation}
and $|\tau-\tau_0|\leq K_3\sqrt{N}|y-\alpha|$, where $K_3$ is a positive constant depending only on $y$. Since $a,b,c, d$ are chosen as in (\ref{equation:condition-on-homography}), (\ref{equation:tau0-is-CM}) shows that $j(\tau_0)$ is the $j$-invariant of an elliptic curve $E_0$ with a cyclic self-isogeny of degree $N$. Now writing 
$$|y-j(E_0)|\leq K_2^{-1}|\tau-\tau_0|\leq K_2^{-1}K_3\sqrt{N}|y-\alpha|$$
concludes the proof.
\end{proof}

As in the non-archimedean case, we get the following.

\begin{corollary}\label{corollary:finding-CM-archimedean-2}
Let $y$ and $z$ be two elements of $\C\subset X(1)(\C)$. Then there exist positive constants $C', \varepsilon'$ such that for any two distinct positive integers $N_1, N_2>1$, and any $(\alpha, \beta)\in T_{N_1*}y\times T_{N_2*}y$ such that 
$$\max(|\alpha-z|, |\beta-z|)\leq \frac{\varepsilon'}{N_1N_2},$$
there exists a $CM$ elliptic curve $E_0$ over $\C$ with a cyclic self-isogeny of degree at most $N_1N_2$ such that 
$$|y-j(E_0)|\leq C'\sqrt{N_1N_2}|\alpha-\beta|.$$
\end{corollary}

\begin{proof}
Since $\alpha$ belongs to $T_{N_1*}y$, $y$ belongs to $T_{N_1*}\alpha$. Furthermore, since $N_1$ and $N_2$ are distinct, the elements of $T_{N_1*}\beta$ are all elements of some $T_{N*}y$ for some positive integer $N$ with $1<N\leq N_1N_2$.

As a consequence of Proposition \ref{proposition:finding-CM-archimedean-1}, it is enough to show that, if $|\alpha-z$ and $|\beta-z|$ are smaller than a constant depending only on $y$ and $z$, there exists $\beta'\in T_{N_1*}\beta$ such that $|y-\beta'|\leq K_3|\alpha-\beta|$, where $K_3$ is a positive constant depending only on $y$ and $z$.

Since $y$ is the Hecke orbit $T_{N_1*}\alpha$, we can write $y=j(\frac{a\tau_{\alpha}+b}{c\tau_\alpha+d})$, where $\tau_{\alpha}$ is an element of $\mathbb H$ with $j(\tau_\alpha)=\alpha$ and $a,b,c,d$ are as in (\ref{equation:condition-on-homography}) -- $N$ being replaced by $N_1$ of course. Now we can find $\tau_{\beta}\in \mathbb H$ with $j(\tau_\beta)=\beta$ and $|\tau_\alpha-\tau_\beta|\leq K_4 |\alpha-\beta|$ for some positive constant $K_4$ depending only on $z$. 

Since homographies preserve the hyperbolic distance on $\mathbb H$, the hyperbolic distance between $y$ and $\frac{a\tau_\beta+b}{c\tau_{\beta}+d}$ is equal to that between $\alpha$ and $\beta$. Writing $\beta'=j(\frac{a\tau_\beta+b}{c\tau_{\beta}+d})$ and noting that the hyperbolic distance on $\mathbb H$ and the usual distance on $\C\subset X(1)(\C)$ are equivalent via $j$ on neighborhoods of $\tau$ and $j(\tau)=y$, this shows the inequality 
$$|y-\beta'|\leq K_3|\alpha-\beta|$$
where $K_3$ depends only on $y$ and $z$. By construction, $\beta'$ belongs to $T_{N_1*}\beta$, which allows us to conclude.
\end{proof}

The following easy result shows that $CM$ point cannot be too close to one another.

\begin{proposition}\label{proposition:CM-far-apart-archimedean}
Let $K$ be a compact subset of $\C\subset X(1)(\C)$. Let $M_1, M_2>1$ be two integers, and let $E_1, E_2$ be two $CM$ elliptic curves over $\C$ with cyclic self-isogenies of degree $M_1$ and $M_2$ respectively. Assume that $j(E_1)$ and $j(E_2)$ belong to $K$. If $E_1$ and $E_2$ are not isomorphic, then 
$$|j(E_1)-j(E_2)|\geq C(M_1M_2)^{-1/2}(\sqrt{M_1}+\sqrt{M_2})^{-1},$$
where $C$ is a positive constant depending only on the compact set $K$. 
\end{proposition}

\begin{proof}
As before, and since we are working over a compact set, we only have to prove that if $K'$ is any compact subset of $\mathbb H$, then for any $\tau_1, \tau_2\in \mathbb K$ such that $j(\tau_i)=E_i$ for $i=1,2$, we have 
$$|\tau_1-\tau_2|\geq D(M_1M_2)^{-1/2}(\sqrt{M_1}+\sqrt{M_2})^{-1}$$
for some positive constant $D$ depending only on the compact set $K$. 

Since $E_1$ has a cycle self-isogeny of degree $M_1$, we have
$$\tau_1=\frac{a_1\tau_1+b_1}{c_1\tau_1+d_1},$$
where the matrix 
$$\begin{pmatrix}
a_1 & b_1\\
c_1 & d_1
\end{pmatrix}\in \mathrm{M_2(\Z)}$$ 
has determinant $M_1$ and is not a homothety. In particular, $c_1\neq 0$. By Lemma \ref{lemma:bounding-coefficients}, $|c_1|$ is bounded above by $K_1\sqrt{N}$, where the positive constant $K_1$ only depends on the compact set $K'$. Now we can write
$$\tau_1=\frac{\lambda_1+\mu_1\sqrt{t^2-4M_1}}{2c_1},$$
where $t=a_1+d_1$ and $\lambda_1, \mu_1$ are two integers. Note that $t^2<4M_1$. Writing $\tau_2$ in the same fashion and computing $\tau_1-\tau_2$ yields the result -- noting that $|\sqrt{a}-\sqrt{b}|\geq \frac{1}{\sqrt{a}+\sqrt{b}}$ for any two distinct positive integers $a$ and $b$.
\end{proof}

Putting the estimates of Corollary \ref{corollary:finding-CM-archimedean-2} and Proposition \ref{proposition:CM-far-apart-archimedean} together, we obtain the following statement, the proof of which we leave to the reader, as it is completely analogous to that of Proposition \ref{proposition:bad-approximation-p-adic}

\begin{proposition}\label{proposition:bad-approximation-archimedean}
Let $D$ be a large enough positive integer. For any $y, z\in \C\subset X(1)(\C)$, define
$$S_{y, z}=\{N\in\mathbb N | \exists\alpha\in T_{N*}y, |\alpha-z|\leq N^{-D}\}.$$
Assume that $y$ has nonnegative valuation and is not the $j$-invariant of a CM elliptic curve. Then $S_{y, z}$ has density zero, i.e. 
$$\lim_{n\ra\infty} \frac{1}{n}|S_{y,z}|=0.$$
\end{proposition}

Unwinding the estimates above, it can be checked that the statement above is effective. The reader will be able to verify that one can take $D=4$, though it is by no means an optimal choice.

\bibliographystyle{alpha}
\bibliography{isogenies}

\begin{thebibliography}{BKPSB98}

\bibitem[Aut03]{Autissier03}
P.~Autissier.
\newblock Hauteur des correspondances de {H}ecke.
\newblock {\em Bull. Soc. Math. France}, 131(3):421--433, 2003.

\bibitem[BKPSB98]{BKPSB}
R.~Borcherds, L.~Katzarkov, T.~Pantev, and N.~Shepherd-Barron.
\newblock Families of {K}3 surfaces.
\newblock {\em J. Algebr. Geom.}, 7:183--193, 1998.

\bibitem[Bos99]{Bost99}
J.-B. Bost.
\newblock Potential theory and {L}efschetz theorems for arithmetic surfaces.
\newblock {\em Ann. Scient. \'{E}c. Norm. Sup.}, 32:241--312, 1999.

\bibitem[CHT08]{ClozelHarrisTaylor08}
L.~Clozel, M.~Harris, and R.~Taylor.
\newblock Automorphy for some {$l$}-adic lifts of automorphic mod {$l$}
  {G}alois representations.
\newblock {\em Publ. Math. Inst. Hautes \'Etudes Sci.}, (108):1--181, 2008.
\newblock With Appendix A, summarizing unpublished work of Russ Mann, and
  Appendix B by Marie-France Vign{\'e}ras.

\bibitem[CO06]{ChaiOort06}
C.-L. Chai and F.~Oort.
\newblock Hypersymmetric abelian varieties.
\newblock {\em Pure Appl. Math. Q.}, 2(1, Special Issue: In honor of John H.
  Coates. Part 1):1--27, 2006.

\bibitem[Coh84]{Cohen84}
P.~Cohen.
\newblock On the coefficients of the transformation polynomials for the
  elliptic modular function.
\newblock {\em Math. Proc. Cambridge Philos. Soc.}, 95(3):389--402, 1984.

\bibitem[Con04]{Conrad04}
B.~Conrad.
\newblock Gross-{Z}agier revisited.
\newblock In {\em Heegner points and {R}ankin {$L$}-series}, volume~49 of {\em
  Math. Sci. Res. Inst. Publ.}, pages 67--163. Cambridge Univ. Press,
  Cambridge, 2004.
\newblock With an appendix by W. R. Mann.

\bibitem[COU01]{ClozelOhUllmo01}
L.~Clozel, H.~Oh, and E.~Ullmo.
\newblock Hecke operators and equidistribution of {H}ecke points.
\newblock {\em Inventiones mathematicae}, 144(2):327--351, 2001.

\bibitem[Elk89]{Elkies89}
N.~Elkies.
\newblock Supersingular primes for elliptic curves over real number fields.
\newblock {\em Compositio Mathematica}, 72(2):165--172, 1989.

\bibitem[Fal83]{Faltings83}
G.~Faltings.
\newblock Endlichkeitss\"atze f\"ur abelsche {V}ariet\"aten \"uber
  {Z}ahlk\"orpern.
\newblock {\em Invent. Math.}, 73(3):349--366, 1983.

\bibitem[Gro86]{Gross86}
B.~H. Gross.
\newblock On canonical and quasicanonical liftings.
\newblock {\em Invent. Math.}, 84(2):321--326, 1986.

\bibitem[GZ85]{GrossZagier85}
B.~H. Gross and D.~B. Zagier.
\newblock On singular moduli.
\newblock {\em J. Reine Angew. Math.}, 355:191--220, 1985.

\bibitem[GZ86]{GrossZagier86}
B.~H. Gross and D.~Zagier.
\newblock Heegner points and derivatives of {L}-series.
\newblock {\em Inventiones mathematicae}, 84(2):225--320, 1986.

\bibitem[HSBT10]{HarrisShepherBarronTaylor10}
M.~Harris, N.~Shepherd-Barron, and R.~Taylor.
\newblock A family of {C}alabi-{Y}au varieties and potential automorphy.
\newblock {\em Ann. of Math. (2)}, 171(2):779--813, 2010.

\bibitem[Kud14]{Kudla14}
S.~Kudla.
\newblock A note about special cycles on moduli spaces of {K}3 surfaces.
\newblock arXiv:1408.1907 [math.AG], 2014.

\bibitem[LT76]{LangTrotter76}
S.~Lang and H.~Trotter.
\newblock {\em Frobenius distributions in {${\rm GL}_{2}$}-extensions}.
\newblock Lecture Notes in Mathematics, Vol. 504. Springer-Verlag, Berlin-New
  York, 1976.
\newblock Distribution of Frobenius automorphisms in
  ${{\rm{G}}L}_{2}$-extensions of the rational numbers.

\bibitem[Ogu03]{Oguiso03}
K.~Oguiso.
\newblock Local families of {K}3 surfaces and applications.
\newblock {\em J. Alg. Geom.}, 12(3):405--433, 2003.

\bibitem[Sil90]{Silverman90}
Joseph~H. Silverman.
\newblock Hecke points on modular curves.
\newblock {\em Duke Math. J.}, 60(2):401--423, 1990.

\bibitem[Tat66]{Tate66}
J.~Tate.
\newblock {Endomorphisms of Abelian varieties over finite fields.}
\newblock {\em Invent. Math.}, 2:134--144, 1966.

\bibitem[Tay08]{Taylor08}
R.~Taylor.
\newblock Automorphy for some {$l$}-adic lifts of automorphic mod {$l$}
  {G}alois representations. {II}.
\newblock {\em Publ. Math. Inst. Hautes \'Etudes Sci.}, (108):183--239, 2008.

\bibitem[Voi02]{Voisin2002}
C.~Voisin.
\newblock {\em Th\'eorie de {H}odge et g\'eom\'etrie alg\'ebrique complexe},
  volume~10 of {\em Cours Sp\'ecialis\'es}.
\newblock Soci\'et\'e Math\'ematique de France, Paris, 2002.

\end{thebibliography}

\end{document}